\documentclass[11pt,letterpaper]{amsart}
\usepackage[usenames,dvipsnames]{color}
\usepackage{amsthm,amsfonts,amssymb,amsmath,amsxtra}
\usepackage[all]{xy}
\SelectTips{cm}{}
\usepackage{xr-hyper}
\usepackage[colorlinks=
   citecolor=Black,
   linkcolor=Red,
   urlcolor=Blue]{hyperref}
\usepackage{verbatim}
\usepackage{caption}

\usepackage{lineno}

\usepackage{mathrsfs}

\RequirePackage{xspace}
\RequirePackage{etoolbox}
\RequirePackage{varwidth}
\RequirePackage{enumitem}
\RequirePackage{tensor}
\RequirePackage{mathtools}
\RequirePackage{longtable}
\RequirePackage{multirow}

\def\ge{\geqslant}
\def\le{\leqslant}
\def\a{\alpha}
\def\b{\beta}
\def\g{\gamma}
\def\G{\Gamma}
\def\d{\delta}

\def\e{\epsilon}

\def\o{\omega}

\def\s{\sigma}

\def\k{\kappa}
\def\l{\lambda}

\def\i{^{-1}}

\def\<{\langle}
\def\>{\rangle}
\def\bq{\mathbf q}

\newcommand{\bG}{\mathbf G}

\newcommand{\bM}{\mathbf M}

\newcommand{\BF}{\ensuremath{\mathbb {F}}\xspace}
\newcommand{{\BG}}{\ensuremath{\mathbb {G}}\xspace}

\newcommand{{\BK}}{\ensuremath{\mathbb {K}}\xspace}

\newcommand{\BQ}{\ensuremath{\mathbb {Q}}\xspace}
\newcommand{\BR}{\ensuremath{\mathbb {R}}\xspace}
\newcommand{\BS}{\ensuremath{\mathbb {S}}\xspace}

\newcommand{\BZ}{\ensuremath{\mathbb {Z}}\xspace}

\newcommand{\CK}{\ensuremath{\mathcal {K}}\xspace}

\newcommand{\CO}{\ensuremath{\mathcal {O}}\xspace}

\newcommand{\CR}{\ensuremath{\mathcal {R}}\xspace}

\newcommand{\GL}{\mathrm{GL}}

\newcommand{\irr}{\mathrm{irr}}
\newcommand{\leng}{\mathrm{length}}

\let\Im\relax
\DeclareMathOperator{\Im}{Im}

\def\leng{{\rm length}}
\def\bG{\mathbf{G}}
\def\bB{\mathbf{B}}
\def\bT{\mathbf{T}}
\def\bU{\mathbf{U}}
\def\bN{\mathbf{N}}

\def\bZ{\mathbf{Z}}

\def\ts{{\tilde s}}
\def\brk{{\breve k}}
\def\dw{{\dot w}}

\def\dc{{\dot c}}
\def\COk{{\CO_{\brk}}}

\def\pr{{\rm pr}}
\def\Cross{{\rm Cross}}

\def\codim{{\rm codim}}

\newtheorem{theorem}{Theorem}
\newtheorem{proposition}[theorem]{Proposition}
\newtheorem{lemma}[theorem]{Lemma}

\newtheorem{corollary}[theorem]{Corollary}

\theoremstyle{definition}
\newtheorem{definition}[theorem]{Definition}

\newtheorem{remark}[theorem]{Remark}

\numberwithin{equation}{section}
\numberwithin{theorem}{section}

\setitemize[0]{leftmargin=*,itemsep=\the\smallskipamount}
\setenumerate[0]{leftmargin=*,itemsep=\the\smallskipamount}

\renewcommand{\to}{%
   \ifbool{@display}{\longrightarrow}{\rightarrow}%
   }
\let\shortmapsto\mapsto
\renewcommand{\mapsto}{%
   \ifbool{@display}{\longmapsto}{\shortmapsto}%
   }
\newlength{\olen}
\newlength{\ulen}
\newlength{\xlen}
\newcommand{\xra}[2][]{%
   \ifbool{@display}%
      {\settowidth{\olen}{$\overset{#2}{\longrightarrow}$}%
       \settowidth{\ulen}{$\underset{#1}{\longrightarrow}$}%
       \settowidth{\xlen}{$\xrightarrow[#1]{#2}$}%
       \ifdimgreater{\olen}{\xlen}%
          {\underset{#1}{\overset{#2}{\longrightarrow}}}%
          {\ifdimgreater{\ulen}{\xlen}%
             {\underset{#1}{\overset{#2}{\longrightarrow}}}
             {\xrightarrow[#1]{#2}}}}%
      {\xrightarrow[#1]{#2}}
   }
\makeatother
\newcommand{\xyra}[2][]{%
   \settowidth{\xlen}{$\xrightarrow[#1]{#2}$}%
   \ifbool{@display}%
      {\settowidth{\olen}{$\overset{#2}{\longrightarrow}$}%
       \settowidth{\ulen}{$\underset{#1}{\longrightarrow}$}%
       \ifdimgreater{\olen}{\xlen}%
          {\mathrel{\xymatrix@M=.12ex@C=3.2ex{\ar[r]^-{#2}_-{#1} &}}}%
          {\ifdimgreater{\ulen}{\xlen}%
             {\mathrel{\xymatrix@M=.12ex@C=3.2ex{\ar[r]^-{#2}_-{#1} &}}}
             {\mathrel{\xymatrix@M=.12ex@C=\the\xlen{\ar[r]^-{#2}_-{#1} &}}}}}%
      {\mathrel{\xymatrix@M=.12ex@C=\the\xlen{\ar[r]^-{#2}_-{#1} &}}}%
   }
\makeatletter
\newcommand{\xla}[2][]{%
   \ifbool{@display}%
      {\settowidth{\olen}{$\overset{#2}{\longleftarrow}$}%
       \settowidth{\ulen}{$\underset{#1}{\longleftarrow}$}%
       \settowidth{\xlen}{$\xleftarrow[#1]{#2}$}%
       \ifdimgreater{\olen}{\xlen}%
          {\underset{#1}{\overset{#2}{\longleftarrow}}}%
          {\ifdimgreater{\ulen}{\xlen}%
             {\underset{#1}{\overset{#2}{\longleftarrow}}}
             {\xleftarrow[#1]{#2}}}}%
      {\xleftarrow[#1]{#2}}
   }
\newcommand{\isoarrow}{%
   \ifbool{@display}{\overset{\sim}{\longrightarrow}}{\xrightarrow\sim}%
   }

\begin{document}
\title{Steinberg's cross-section of Newton strata}
\author{Sian Nie}
\address{Academy of Mathematics and Systems Science, Chinese Academy of Sciences, Beijing 100190, China}

\address{ School of Mathematical Sciences, University of Chinese Academy of Sciences, Chinese Academy of Sciences, Beijing 100049, China}
\email{niesian@amss.ac.cn}


\begin{abstract}
In this note, we introduce a natural analogue of Steinberg's cross-section in the loop group of a reductive group $\bG$. We show this loop Steinberg's cross-section provides a simple geometric model for the poset $B(\bG)$ of Frobenius-twisted conjugacy classes (referred to as Newton strata) of the loop group. As an application, we confirms a conjecture by Ivanov on decomposing loop Delgine-Lusztig varieties of Coxeter type. This geometric model also leads to new and direct proofs of several classical results, including the converse to Mazur's inequality, Chai's length formula on $B(\bG)$, and a key combinatorial identity in the study affine Deligne-Lusztig varieties with finite Coxeter parts.
\end{abstract}

\maketitle

\section*{Introduction}

\subsection{Steinberg's cross-section}
Let $G$ be a reductive group over an algebraically closed field. Recall that an element of $G$ is called regular if the dimension of its centralizer in $G$ equals the rank $r$ of $G$. In his seminal work \cite{S}, Steinberg constructed a cross-section $N$ to regular (conjugacy) classes of $G$, and proved the following remarkable properties:

(1) $N$ consists of regular elements;

(2) each regular class of $G$ intersects $N$ tranversely;

(3) if $G$ is semisimple and simply connected, each regular class intersects $N$ at a single point.

By construction, the cross-section $N$ is an affine $r$-space attached to a minimal Coxeter element in the Weyl group of $G$. In recent works \cite{HL}, \cite{L11}, \cite{Se}, He, Lusztig, and Sevostyanov discovered that for a wide class of Weyl group elements, analogous construction also produces transversal slices to conjugacy classes of $G$.

\subsection{The main result} In this note we focus on a loop group version of Steinberg's result mentioned above. On the one hand, the Steinberg's cross-section has a straightforward analogue in a loop group. On the other hand, a loop group is stratified by its Frobenius-twisted conjugacy classes. So it is natural to ask how the loop Steinberg's cross-section intersects various Frobenius-twisted conjugacy classes. The answer turns out to be simple and purely combinatorial. This suggests that the loop Steinberg's cross-section serves as a ``transversal slice'' to Frobenius-twisted conjugacy classes.

To formulate the main results, we introduce more notations. Let $\bG$ be an unramified reductive group over a non-archimedean local field $k$. Fix a uniformizer $\varpi \in k$. Let $\brk$ be the completion of a maximal unramified extension of $k$. Denote by $\s$ the Frobenius automorphism $\brk / k$ and the induced automorphism of the loop group $\bG(\brk)$. We fix two opposite Borel subgroups $\bB$ and $\bB^-$, and let $\bT = \bB \cap \bB^-$ be a maximal torus. Let $\bN$ be the normalizer of $\bT$ in $\bG$. The Weyl group of $\bG$ is defined by \[W = \bN(\COk) / \bT(\COk) \cong \bN(\brk) / \bT(\brk),\] where $\COk$ is the integer ring of $\brk$. For $w \in W$ we set \[\bU_w = \bU \cap {}^{w\i}(\bU^-),\] where $\bU, \bU^-$ are the unipotent radicals of $\bB, \bB^-$ respectively.

Fix a representative set $\{\a_1, \dots, \a_r\}$ for the $\s$-orbits of simple roots appearing in $\bB$. Denote by $s_{\a_i} \in W$ the simple reflection corresponding to $\a_i$. Set $c = s_{\a_1} \cdots s_{\a_r} \in W$ which is a $\s$-Coxeter element. Let $\dc \in \bN(\brk)$ be lift of $c$. Then there exist a unique cocharacter $\mu \in X_*(\bT)$ and some lifts $\ts_{\a_i} \in \bN(\COk)$ of $s_{\a_i}$ such that $\dc = \varpi^\mu \ts_{\a_1} \cdots \ts_{\a_r}$, where $\varpi^\mu = \mu(\varpi) \in \bG(k)$. Following \cite[Theorem 1.4]{S}, we define the loop Steinberg's cross-section attached to $\dc$ by \[\dc \bU_c(\brk) = \varpi^\mu \ts_{\a_1} \bU_{\a_1}(\brk) \cdots \ts_{\a_r} \bU_{\a_r}(\brk),\] where $\bU_{\a_i}$ denotes the root subgroup of $\a_i$.

For $b \in \bG(\brk)$ the corresponding $\s$-twisted conjugacy class is defined by \[[b] = \{g\i b \s(g); g \in \bG(\brk)\}.\] Thanks to Kottwitz \cite{K}, $[b]$ is determined by two invariants: the Kottwitz point $\k(b) \in \pi_1(\bG)_\s$ and the dominant Newton point $\nu(b) \in X_*(\bT)_\BQ$, see \S\ref{setup-subsec}. We denote by $B(\bG)$ the set of $\s$-twisted conjugacy classes of $\bG(\brk)$.

The main result of this note is a combinatorial description of the intersections $[b] \cap \dc \bU_c(\brk)$ for $[b] \in B(\bG)$.
\begin{theorem} \label{main}
    Let $b$ and $\dc = \varpi^\mu \ts_{\a_1} \cdots \ts_{\a_r}$ be as above. Then we have $[b] \cap \dc\bU_c(\brk) \neq \emptyset$ if and only if $\k(\dc) = \k(b)$. In this case, \[[b] \cap \dc\bU_c(\brk) = \varpi^\mu \ts_{\a_1} H_{\mu, b, \a_1} \cdots \ts_r H_{\mu, b, \a_r},\] where \[H_{\mu, b, \a_i} = \begin{cases} \bU_{\a_i}(\varpi^{\<\mu-\nu(b),  \o_i\>} \CO_\brk^\times); &\text{ if } \<\nu(b), \a_i\> > 0 \\ \bU_{\a_i} (\varpi^{\lceil \<\mu-\nu(b), \o_i\> \rceil} \COk), &\text{ otherwise.} \end{cases}\] Here $\o_i$ is the sum of fundamental weights corresponding simple roots in the $\s$-orbits of $\a_i$, and $\< , \>: X_*(\bT)_\BQ \times X^*(\bT)_\BQ \to \BQ$ is the natural paring.
\end{theorem}
We remark that $\<\mu-\nu(b), \o_i\> \in \BZ$ if $\<\nu(b), \a_i\> > 0$, see Definition \ref{def}.

\begin{remark}
    Note that $[b]$ is an admissible set in the sense of \cite{HV} and \cite{He16}. In particular, it makes sense to define topological invariants of $[b]$, such as the closure of $[b]$ in $\bG(\brk)$, the irreducible/connected components of $[b]$, and the relative dimension of $[b]$.

    By construction, the loop Steinberg's cross-section $\dc\bU_c(\brk)$ can be viewed as an infinite dimensional affine space. It follows from Theorem \ref{main} that the intersections $[b] \cap \dc\bU_c(\brk)$ are admissible subsets of $\dc\bU_c(\brk)$. So it also makes sense to study their topological properties.
\end{remark}

Let $\leq$ denote the usual dominance order on $B(\bG)$, see \S\ref{subsec:closure}. Our second result shows that the following Newton decomposition \[\dc\bU_c(\brk) = \bigsqcup_{[b] \in B(\bG)} [b] \cap \dc\bU_c(\brk)\] is a stratification of $\dc\bU_c(\brk)$, whose closure relation is given by $\leq$.
\begin{theorem} \label{closure}
    Let $\mu \in X_*(\bT)$ and $[b] \in B(\bG)$ such that $\k(\varpi^\mu) = \k(b)$. Then the closure of $[b] \cap \dc\bU_c(\brk)$ in $\dc\bU_c(\brk)$ is \[ \overline{[b] \cap \dc\bU_c(\brk)} = \bigsqcup_{[b'] \leq [b]} [b'] \cap \dc\bU_c(\brk).\] Moreover, $\overline{[b] \cap \dc\bU_c(\brk)} - ([b] \cap \dc\bU_c(\brk))$ is pure of codimension one in $\overline{[b] \cap \dc\bU_c(\brk)}$.
\end{theorem}

The Newton stratification of $\dc \bU_c(\brk)$ provides a simple geometric model of the poset $(\bG, \leq)$. This lead to new interpretations of several classical results. Here are two immediate examples.

\

The first one is a geometric characterization of the dominance order on $B(\bG)$ proved by He \cite{He16}.
\begin{theorem} \label{order}
    Let $[b], [b'] \in B(\bG)$. Then $[b'] \subseteq \overline{[b]}$ if and only if $[b'] \leq [b]$. Here $\overline{[b]}$ denotes the closure of $[b]$ in $\bG(\brk)$.
\end{theorem}
The implication ($\Rightarrow$) follows from a general result of Rapoport and Richartz \cite{RR}. The implication ($\Leftarrow$) is proved by He \cite{He16} using a deep purity result due to Viehmann \cite{Vie} and Hamacher \cite{Ha}. Now the implication ($\Leftarrow$) is a direct consequence of Theorem \ref{closure}.

\

The second one is the Mazur's inequality criterion on the emptiness/non-emptiness of $[b] \cap \CK \varpi^\mu \CK$, where $\CK=\bG(\COk) \subseteq \bG(\brk)$.
\begin{theorem} \label{Mazur}
    Let $\mu \in X_*(\bT)$ and $[b] \in B(\bG)$. Then $[b] \cap \CK \varpi^\mu \CK \neq \emptyset$ if and only if $[b] \in B(\bG, \mu)$, namely, $[b] \leq [\varpi^\mu]$.
\end{theorem}
The implication ($\Rightarrow$) is again due to Rapoport and Richartz \cite{RR}. The implication ($\Leftarrow$) is a conjecure of Kottwitz and Rapoport, which is proved by Gashi \cite{Ga} in a purely combinatorial way. Now we give a new proof of the implication ($\Leftarrow$). First we can assume that $\mu$ is dominant. Then $\nu(\varpi^\mu)$ equals the $\s$-average $\mu^\diamond$ of $\mu$. By definition, for any $[b] \in B(\bG, \mu)$ we have $\k(b) = \k(\varpi^\mu)$ and $\<\mu-\nu(b), \o_i\> = \<\mu^\diamond -\nu(b), \o_i\> \ge 0$ for $1 \le i \le r$. Hence it follows from Theorem \ref{main} that \[ \emptyset \neq [b] \cap \dc\bU_c(\brk) \subseteq \varpi^\mu \ts_{\a_1} \bU_{\a_1}(\COk) \cdots \ts_{\a_r} \bU_{\a_r}(\COk) \subseteq \varpi^\mu \CK.\] So the implication ($\Leftarrow$) follows.

\subsection{Further consequences} We continue to discuss further consequences of the main results.

In \cite{Chai}, Chai deduced an explicit length formula for the poset $(B(\bG), \leq)$ in a purely combinatorial way. We will give a geometric proof of Chai's formula.
\begin{theorem} \label{poset}
    The poset $(B(\bG), \leq)$ is ranked, and the length between two elements $[b'] \leq [b] \in B(\bG)$ is given by \[\leng([b], [b']) = \sum_{i=1}^r \lceil \<\mu-\nu(b'), \o_{\CO_i}\>\rceil - \lceil \<\mu-\nu(b), \o_{\CO_i}\>\rceil,\] where $\mu \in X_*(\bT)$ is any cocharacter such that $\k(\varpi^\mu) = \k(b) =\k(b')$.
\end{theorem}

\

In \cite{HNY}, He, Yu and the author show that an affine Deligne-Lusztig varieties with a finite Coxeter part has a simple geometric structure, see also \cite{Sh}. A key ingredient of our approach is a combinatorical identity, whose proof is quite involved. We will give a new and direct proof of the identity by counting $k$-rational points on the (truncated) loop Steinberg's cross-section. The same idea also guides us to discover a new combinatorical identity.
\begin{theorem} \label{identity}
    Let $\mu \in X_*(\bT)$. Then we have \[\sum_{[b] \in B(\bG, \mu)} (\bq-1)^{\sharp J_{\nu(b)}} \bq^{- \sharp J_{\nu(b)} - \leng([\varpi^\mu], [b])} = 1,\] where $J_{\nu(b)} = \{1 \le i \le r; \<\nu(b), \a_i\> \neq 0\}$. Moreover, if $B(\bG, \mu)_\irr \neq \emptyset$, then \[\sum_{[b] \in B(\bG, \mu)_\irr} (\bq-1)^{\sharp J_{\nu(b)}} \bq^{r - \sharp J_{\nu(b)} - \leng([\varpi^\mu], [b])} = 1.\] Here $B(\bG, \mu)_\irr$ is the set of Hodge-Newton irreducible elements in $B(\bG, \mu)$, see \S \ref{subsec:identity}.
\end{theorem}

\begin{remark} The first identity is new. The second identity is the one from \cite{HNY}, which is also proved by Lim \cite{Lim} based on a probability-theoretic interpretation.
\end{remark}

The last application is devoted to loop Deligne-Lusztig varieties. These varieties were first introduced by Lusztig \cite{L79}. As in the classical Deligne-Lusztig theory \cite{DL}, the cohomologies of loop Deligne-Lusztig varieties are expected to realize interesting representations of $p$-adic groups. In the case of general linear groups and their inner forms, Boyarchenko, Weinstein, Chan, and Ivanov carried out an extensive study of loop Deligne-Lusztig varieties of Coxeter type see \cite{BW}, \cite{Chan}, \cite{CI}, \cite{CI2} and references therein. As a desirable consequence, they obtained a purely local, purely geometric and explicit realization of local Langlands and Jacquet-Langlands correspondences for a wide class of irreducible supercuspidal representations.

A fundamental step in the works mentioned above is to decompose the loop Deligne-Lusztig variety into a union of translates of parahoric level Deligne-Lusztig varieties. By Theorem \ref{main} we have the following general result.
\begin{theorem} \label{dec}
    The decomposition theorem for loop Deligne-Lusztig varieties of Coxeter type holds for all unramified reductive groups, see \cite[Theorem 1.1]{Iv2} for the precise formulation. In particular, these varieties are representable by perfect schemes.
\end{theorem}

\begin{remark}
    Theorem \ref{dec} is also proved in \cite{IN} using a different method. The last statement of the theorem confirms a conjecture of Ivanov \cite[Conjecture 1.1]{Iv1} in the basic case. If $\bG$ is of classical type, the theorem is proved by Ivanov \cite{Iv2}.
\end{remark}

\subsection{Outline} The note is organized as follows. In \S 1 we introduce basic notions/constructions, and reduce Theorem \ref{main} to certain translation case, see Theorem \ref{st-cross}. In \S 2 we study the intersection $[b] \cap \dc \bT(\COk) \bU_c(\brk)$ when $b$ is a translation, and show it equals the image of certain explicit map, see Proposition \ref{red}. In \S 3 we make a digression to introduce an algorithm to compute the image. In \S 4 we finish the proof of Theorem \ref{main} using the algorithm. In the last section, we prove Theorem 0.6 and Theorem 0.7 as applications of Theorem \ref{main}.

\subsection*{Acknowledgement} We would like to thank Xuhua He, Alexander Ivanov and Qingchao Yu for helpful comments and discussions.

\section{Preliminary}
In this section, we introduce basic notations and reduce Theorem \ref{main} to the translation case, namely the case where $b = \varpi^\chi$ for certain $\chi \in X_*(\bT)$.

\subsection{} \label{setup-subsec}
We keep the notations $\bG, \bB, \bU, \bT, \bN, W$ introduced in the introduction. Let $\Phi^+ \subseteq \Phi$ be the set of (positive) roots of $\bT$ that appears in $\bB$. We have $\Phi = \Phi^+ \sqcup \Phi^-$ with $\Phi^- = - \Phi^+$. Let $\Pi \subseteq \Phi^+$ be the set of simple roots. We say $v \in X_*(\bT)$ is dominant if $\<v, \a\> \ge 0$ for all $\a \in \Phi^+$. Here $\<, \>: X_*(\bT) \otimes X_*(\bT) \to \BQ$ is
the natural pairing. Denote by $\BS \subseteq W$ the corresponding set of simple reflections of $W$. Then $(W, \BS)$ is a Coxeter system.

Let $\a \in \Phi$. We denote by $\a^\vee$ the coroot of $\a$, and let $s_\a \in W$ be the reflection sending $\l$ to $\l - \<\l, \a\> \a^\vee$ for $\l \in X_*(\bT)$. Choose root subgroups $\bU_\a: \BG_a \to \bG$ such that for $z \in \brk^\times$ we have \[\tag{*} \bU_{-\a}(-z\i) \bU_\a(z) \bU_{-\a}(-z\i) = z^{\a^\vee} \tilde s_{-\a} \in \bN(\brk),\] where $\tilde s_{-\a} = \bU_{-\a}(-1) \bU_\a(1) \bU_{-\a}(-1) \in \bN(\CO_{\brk})$.

Let $b \in \bG(\brk)$. Then Kottwitz point $\k(b) = \k_\bG(b)$ is the image of $b$ under the natural projection $\k: \bG(\brk) \to \pi_1(\bG)_\s = \pi_1(G) / (1-\s)\pi_1(\bG)$. To define the Newton point, note that there exists a representative $x = \varpi^\l \dw \in [b]$, where $\l \in X_*(\bT)$ and $\dw \in \bN(\COk)$ is a lift of some element $w \in W$. Let $\nu_x \in X_*(\bT)_\BQ$ be the $(w\s)$-average of $\mu$. Then $\nu(b) = \nu_\bG(b)$ equals the unique dominant $W$-conjugate of the $\nu_x$. Note that $\nu(b) = \s(\nu(b))$ is independent of the choice of $x \in \bN(\brk) \cap [b]$.

\subsection{} \label{cox-subsec}
Let $\{\a_1, \dots, \a_r\}$ be a representative set for the $\s$-orbits of $\Pi$. Let $c = s_{\a_1} s_{\a_2} \cdots s_{\a_r} \in W$, which is a minimal $\s$-Coxeter element.

Let  $\xi = \sum_{i=1}^r n_i \a_i^\vee$ with $n_i \in \BZ$. We define \[K_\xi = \bU_{-\a_1}(\varpi^{m_1} \CO_\brk^\times) \cdots \bU_{-\a_r}(\varpi^{m_r} \CO_\brk^\times),\] where $m_i = -n_i - \sum_{j = i+1}^r n_j \<\a_j^\vee, \a_i\>$ for $1 \le i \le r$.

\begin{lemma} \label{mu-eta}
    Let $\mu, \eta \in X_*(\bT)$. Let $\dc \in \varpi^\mu \bN(\COk)$ be a lift of $c$. Then \[ \varpi^\eta \bT(\COk) \bU_{-\a_1}(\brk^\times) \cdots \bU_{-\a_r}(\brk^\times) \cap \bU(\brk) \dc \bT(\COk) \bU(\brk) = \varpi^\eta \bT(\COk) K_{\mu-\eta}.\] Here we set $K_{\mu-\eta} = \emptyset$ if $\mu - \eta \notin \sum_{i=1}^r \BZ \a_i^\vee$.
\end{lemma}
\begin{proof}
    We argue by induction on $r$. If $r = 0$ the statement is trivial. Let $z_1, \dots, z_r \in \brk^\times$ such that \[\tag{a} \varpi^\eta \bU_{-\a_1}(z_1) \cdots \bU_{-\a_r}(z_r) \in \bU(\brk) \dc \bT(\COk) \bU(\brk).\] Let $n_r \in \BZ$ such that $z_r \in \varpi^{-n_r} \CO_\brk^\times$. By \S \ref{setup-subsec} (*) we have \begin{align*}
        &\quad\ \varpi^\eta \bU_{-\a_1}(z_1) \cdots \bU_{-\a_r}(z_r) \\ &= \varpi^\eta \bU_{-\a_1}(z_1) \cdots \bU_{-\a_{r-1}}(z_{r-1}) \bU_{\a_r}(z_r\i) z_r^{-\a_r^\vee} \tilde s_{\a_r} \bU_{\a_r}(z_r\i) \\ &\subseteq \bU_{\a_r}(\brk) \varpi^{\eta + n_r \a_r^\vee} \bT(\COk) \bU_{-\a_1}(z_r^{-\<\a_r^\vee, \a_1\>} z_1) \cdots \bU_{-\a_{r-1}}(z_r^{-\<\a_r^\vee, \a_{r-1}\> }z_{r-1}) \tilde s_{\a_r} \bU_{\a_r}(\brk).
    \end{align*}

    As the simple reflections $s_{\a_i}$ are distinct,  (a) is equivalent to the following inclusion \[\varpi^{\eta + n_r \a_r^\vee} \bU_{-\a_1}(z_r^{-\<\a_r^\vee, \a_1\>} z_1) \cdots \bU_{-\a_{r-1}}(z_r^{-\<\a_r^\vee, \a_{r-1}\> }z_{r-1}) \in \bU(\brk) \dc \tilde s_{\a_r} \bT(\COk) \bU(\brk).\] Now the statement follows by induction hypothesis.
\end{proof}

\subsection{} \label{H-subsec}
For $w \in W$ we set $\Phi_{w\s}^+ = \Phi^+ \cap (w\s)\i(\Phi^-)$. Then $\Phi_{c\s}^+  = \{\b_1, \dots, \b_r\}$, where $\b_i = \s\i s_{\a_r} \cdots s_{\a_{i+1}}(\a_i) \in \Phi^+$ for $1 \le i \le r$. Note that $\bU_c = \prod_{i=1}^r \bU_{\s(\b_i)}$.

\begin{lemma} \label{commute}
    The root subgroups $\bU_{\b_i}$ for $1 \le i \le r$ commute with each other.
\end{lemma}
\begin{proof}
Assume otherwise. Then $\b_i + \b_{i'} \in \Phi_{c\s}^+$ for some $1 \le i < i' \le r$. So $\b_i + \b_{i'} = \b_{i''}$ for some $1 \le i'' \le r$. As $c$ is product of distinct simple reflections, it follows that \[\s(\b_l) \in \a_l + \sum_{l+1 \le j \le r} \BZ_{\ge 0} \a_j.\] Thus $i'' = i$, which is impossible.
\end{proof}

Let $J \subseteq \{1, \dots, r\}$ and $\l \in X_*(\bT)_\BQ$ such that $\<\l, \b_i\> \in \BZ$ for $i \in J$. We define \[H_{c\s, J}^\l = \prod_{i \in J} \bU_{\b_i}(\varpi^{\<\l, \b_i\>}\CO_\brk^\times) \prod_{i \in  \{1, \dots, r\} -  J} \bU_{\b_i}(\varpi^{\lceil \<\l, \b_i\> \rceil} \COk),\] which is independent of the choice of the product order by Lemma \ref{commute}.

For a simple root $\a \in \Pi$ let $\o_\a \in \BR\Phi$ be the fundamental weights corresponding to $\a$. For $1 \le i \le r$ let $\o_i = \sum_\a \o_\a$, where $\a$ ranges over the $\s$-orbit of $\a_i$. Let $\Phi^\vee \subseteq X_*(\bT)$ denote the set of coroots of $\bT$ in $\bG$.
\begin{lemma} \label{eq}
    Let $\mu \in X_*(\bT)$ and $b \in G(\brk)$ such that $\k(\varpi^\mu) = \k(b)$. Then there exists a unique vector $\l \in X_*(\bT)_\BQ$ modulo $X_*(\bZ)_\BQ^\s$ such that $\mu - \nu(b) = \l - c\s(\l)$, where $\bZ$ is the center of $\bG$. Moreover,

    (1) $\<\l, \g\> + \<\mu, c\s(\g)\>  - \<\l, c\s(\g)\> = \<\nu(b), c\s(\g)\>$ for $\g \in \Phi$;

    (2) $\<\l, \b_i\> = \<\mu - \nu(b), \o_i\>$ for $1 \le i \le r$;

    (3) $\<\mu, \a_1\> - \<\l, \b_1\> - \<\l, \a_1\> = \<\nu(b), \a_1\>$.
\end{lemma}
\begin{proof}
    As $\k(\varpi^\mu) = \k(b)$ we have $\mu - \nu(b) \in \BQ \Phi^\vee$. As  $1-c\s$ restricts to an automorphism of $\BQ\Phi^\vee$, there exists a unique $\l \in \BQ\Phi^\vee$ such that \[\tag{a} \mu - \nu(b) = \l - c\s(\l).\] The uniquness of $\l$ follows from that $(1-c\s)$ preserves the each factor of the direct sum $X_*(\bT)_\BQ = \BQ \Phi^\vee \oplus X_*(\bZ)_\BQ$.

    Now (1) follows directly from (a). Note that $c\s(\l) = \s(\l) - \sum_{l=1}^r \<\l, \b_l\> \a_l^\vee$ and $\s(\o_i) = \o_i$. Thus, (2) follows from that \[ \<\mu - \nu(b), \o_i\> = \<\l - \s(\l), \o_i\> + \sum_{l=1}^r \<\l, \b_l\> \<\a_l^\vee, \o_i\> = \<\l, \b_i\>.\] Finally, (3) follows from (1) by noticing that $-\a_1 = c\s(\b_1)$.
\end{proof}

\begin{definition} \label{def}
Let $\mu, b, \l$ be as in Lemma \ref{eq}. Set \[J_{\nu(b)} = \{1 \le i \le r; \<\nu(b), \a_i\> \neq 0\}.\] By \cite[Lemma 3.5]{HN}, \[\<\l, \b_i\> = \<\mu - \nu(b), \o_i\> \in \BZ \text{ for } i \in J_{\nu(b)}.\] We define \[H_{c, \mu, b} = {}^\s H_{c\s, J_{\nu(b)}}^\l \subseteq \bU_c(\brk),\] which is independent of the choice of $\l \in X_*(\bT)$.
\end{definition}

Now Theorem \ref{main} amounts to the following result.
\begin{theorem} \label{main'}
    Let $\mu \in X_*(\bT)$ and $b \in \bG(\brk)$ such that $\k(\varpi^\mu) = \k(b)$. Let $\dc \in \varpi^\mu \bN(\COk)$ be a lift of $c$. Then \[[b] \cap \dc \bU_c(\brk) = \dc H_{c, \mu, b} = \dc \ts_{\a_1} H_{\mu, b, \a_1} \cdots \ts_{\a_r} H_{\mu, b, \a_r},\] where \[H_{\mu, b, \a_i} = \begin{cases} \bU_{\a_i}(\varpi^{\<\mu-\nu(b),  \o_i\>} \CO_\brk^\times); &\text{ if } i \in J_{\nu(b)}\\ \bU_{\a_r}(\varpi^{\lceil \<\mu-\nu(b), \o_i\> \rceil} \COk), &\text{ otherwise.} \end{cases}\]
\end{theorem}

\subsection{}
The following result is a special case of Theorem \ref{main'}.
\begin{theorem} \label{st-cross}
    Let $\mu, b, \l$ be as in Lemma \ref{eq}. Let $\dc \in \varpi^\mu \bN(\COk)$ be a lift of $c$. Suppose that $\l \in X_*(\bT)$. Then \[ [b] \cap \dc \bT(\COk) \bU_c(\brk) = \dc \bT(\COk) H_{c, \mu, b},\] that is, $[b] \cap \dc \bU_c(\brk) = \dc H_{c, \mu, b}$.
\end{theorem}

Now we prove Theorem \ref{main'} using Theorem \ref{st-cross}.
\begin{proof}[Proof of Theorem \ref{main'}]
Let $\l \in X_*(\bT)_\BQ$ such that $\mu - \nu(b) = \l - c\s(\l)$. Choose $N \in \BZ_{\ge 1}$ such that $\l \in \frac{1}{N} X_*(\bT)$. Let $\brk_N = \brk[\varpi^{1/N}]$, and denote by $[b]_N$ the $\bG(\brk_N)$-$\s$-conjugacy class of $b$. Applying Theorem \ref{st-cross} to the field $\brk_N$ we have \[[b]_N \cap \dc \bU_c(\brk_N) = \dc H_{c, \mu, b, N},\] where $H_{c, \mu, b, N}$ is defined similarly as $H_{c, \mu, b}$ by replacing $\brk$ with $\brk_N$. Thus \[[b] \cap \dc \bU_c(\brk) \subseteq [b]_N \cap \dc \bU_c(\brk_N) \cap \dc \bU_c(\brk)  = \dc H_{c, \mu, b, N} \cap \dc \bU_c(\brk)  = \dc H_{c, \mu, b}.\] On the other hand, for any $x \in \dc H_{c, \mu, b} \subseteq [b]_N$ we have $\nu(x) = \nu(b)$ and $\k(x) = \k(\dc) = \k(b)$. So $x \in [b] \cap \dc \bU_c(\brk)$, and the equality $[b] \cap \dc \bU_c(\brk) = \dc H_{c, \mu, b}$ follows.
\end{proof}

\section{The translation case}
The aim of this section is to study the intersection $[b] \cap \dc\bT(\COk)\bU_c(\brk)$ when $b = \varpi^\chi$ for some $\chi \in X_*(\bT)$. Let $w_0 \in W$ be the longest element.

First we recall two statements which are essentially proved in Proposition (2.2) and Corollary (2.5) of \cite{L76} respectively.
\begin{proposition} \label{cox}
    We have

        (1) ${}^{w_0}\bU(\brk) \cap \bB(\brk) c \bB(\brk) = \bU_{-\a_1}(\brk^\times) \cdots \bU_{-\a_r}(\brk^\times)$;

        (2) if $g \in \bG(\brk)$ and $b \in \bT(\brk)$ satisfies that $g\i b \s(g) \in \bB(\brk) c \bB(\brk)$, then $g \in \bB(\brk) w_0 \bB(\brk)$.
\end{proposition}

The next result is Proposition 8.9 of \cite{S}. See \cite{HL} and \cite{M} for generalizations.
\begin{theorem} \label{crossing}
The map $(x, y) \mapsto x\i y \s(x)$ induces a bijection \[ \Psi_\dc: \bU(\brk) \times \dc \bT(\COk)\bU_c(\brk) \overset \sim \to \bU(\brk) \dc \bT(\COk) \bU(\brk).\]
\end{theorem}

\begin{lemma} \label{unique}
    Let $\mu, \chi \in X_*(\bT)$ such that $\k(\varpi^\mu) = \k(\varpi^\chi)$. Then there exists a unique cocharacter $\eta \in X_*(\bT)$ such that $\mu - \eta \in \sum_{i=1}^r \BZ \a_i^\vee$ and $\chi - \eta \in (1-\s) X_*(\bT)$.
\end{lemma}
\begin{proof}
    The statement follows from the natural isomorphism $\sum_{i=1}^r \BZ \a_i^\vee \cong \BZ\Phi^\vee / ((1-\s)\BZ\Phi^\vee)$.
\end{proof}

\begin{proposition} \label{red}
    Let $b = \varpi^\chi$ for some $\chi \in X_*(\bT)$ and let $\dc \in \varpi^\mu \bN(\COk)$ be a lift of $c$ for some $\mu \in X_*(\bT)$. Then $[b] \cap \dc \bU_c(\brk) \neq \emptyset$ if and only if $\k(\dc) = \k(b)$. In this case, \[\tag{a}[b] \cap \dc \bT(\COk) \bU_c(\brk) = \pr_2 \Psi_{\dc}\i(\varpi^\eta \bT(\COk) K_{\mu-\eta}),\] where $\pr_2: \bU(\brk) \times \dc \bT(\COk) \bU_c(\brk) \to \dc \bT(\COk) \bU_c(\brk)$ is the natural projection and $\eta \in X_*(\bT)$ is the unique cocharacter such that $\mu - \eta \in \sum_{i=1}^r \BZ \a_i^\vee$ and $\chi - \eta \in (1-\s)(X_*(\bT))$ as in Lemma \ref{unique}.
\end{proposition}
\begin{proof}
    As $[b] \cap \dc \bU_c(\brk) \neq \emptyset$ implies that $\k(b) = \k(\varpi^\chi)$, it suffices to deal with the case that $\k(\varpi^\chi) = \k(b) = \k(\dc) = \k(\varpi^\mu)$. Let $b' = \varpi^{w_0(\chi)}$, where $w_0 = \s(w_0)$ is the longest element of $W$. Then $[b'] = [b]$.

    Let $g \in \bG(\brk)$ such that $g\i b' \s(g) \in \dc \bT(\COk) \bU_c(\brk)$. By Proposition \ref{cox} (2), $g \in \bB(\brk) w_0 \bB(\brk)$. Write $g = u \dw_0 u'$ for some $u, u' \in \bU(\brk)$ and some lift $\dw_0 \in \bN(\brk)$ of $w_0$. Then \[h := \dw_0\i u\i b' \s(u) \s(\dw_0) \in \bU(\brk) \dc \bT(\COk) \bU(\brk).\] As $b' = \varpi^{w_0(\chi)}$ and $\s(w_0) = w_0$, it follows that $h \in \varpi^\eta \bT(\COk) {}^{w_0} \bU(\brk)$ for some $\eta \in \chi + (1 - \s)X_*(\bT)$. By Proposition \ref{cox} (1) we have \[h \in \varpi^\eta \bT(\COk) \bU_{-\a_1}(\brk^\times) \cdots \bU_{-\a_r}(\brk^\times).\] By Lemma \ref{mu-eta} we have $\mu - \eta \in \sum_{i=1}^r \BZ \a_i^\vee$ and $h \in \varpi^\eta \bT(\COk) K_{\mu - \eta}$. In particular, $\eta$ is uniquely determined as in Lemma \ref{unique}.

    As $u' \in \bU(\brk)$ and $g\i b' \s(g) = {u'}\i h \s(u')$, it follows by definition that $\pr_2 \Psi_{\dc}\i(h) = g\i b' \s(g) \in \dc \bT(\COk) \bU(\brk)$ and hence \[[b] \cap \dc \bT(\COk) \bU_c(\brk) \subseteq  \pr_2 \Psi_{\dc}\i(\varpi^\eta \bT(\COk) K_{\mu-\eta}).\]

    Conversely, by Lemma \ref{mu-eta} and that $\varpi^\mu \bT(\COk) \bU(\brk) \subseteq [\varpi^\eta] = [b]$ we have \[\varpi^\eta \bT(\COk) K_{\mu-\eta} \subseteq [\varpi^\eta] \cap \bU(\brk) \dc \bT(\COk) \bU(\brk)  = [b] \cap  \bU(\brk) \dc \bT(\COk) \bU(\brk).\] So $[b] \cap \dc \bT(\COk) \bU_c(\brk) \supseteq  \pr_2 \Psi_{\dc}\i(\varpi^\eta \bT(\COk) K_{\mu-\eta})$ and (a) follows.
\end{proof}

\section{A digression}
In this section, we digress to introduce a method to compute the right hand side of Proposition \ref{red} (a). For simplicity we assume that $k = \BF_q((\varpi))$, where $\BF_q$ is a finite field with $q$ elements.

\subsection{} \label{R-sec}
Let $\l \in X_*(\bT)$ and $J \subseteq \{1, \dots, r\}$. By definition, each point $x = (x_i^l)_{i, l}$ in $H_{c\s, J}^\l$ can be uniquely written as \[ x = \prod_{i=1}^r \bU_{\b_i}(\varpi^{\<\l, \b_i\>} \sum_{l=0}^\infty x_i^l \varpi^l),\] where $x_i^l \in \overline \BF_q^\times$ if $l=0$ and $i \in J$, and $x_i^l \in \overline \BF_q$ otherwise. Let $(X_i^l)_{i, l}$ be the coordinates of the points $x = (x_i^l)_{i, l} \in H_{c\s, J}^\l$. Set $\CR = R((\varpi))$, where \[R = \overline \BF_q[X_i^l; l \in \BZ_{\ge 0}, \ 1 \le i \le r] [(X_j^0)\i; j \in J].\]

Let $\e \in \BR$. Let $\CR^{\ge \e}$ be the linear space of elements $\sum_l f^l \varpi^l \in \CR$ such that
\begin{itemize}
    \item $f^l = 0$ if $l < \e$;

    \item $f^l \in \overline \BF_q[X_i^{l'}; 0 \le l' \le l-\e, 1 \le i \le r][(X_j^0)\i; j \in J]$ for $l \ge \e$.
\end{itemize}
 Similarly, let $\CR^{> \e}$ be the linear space of elements $\sum_l f^l \varpi^l$ such that
 \begin{itemize}
     \item $f^l = 0$ if $k \le \e$;

     \item $f^l \in \overline \BF_q[X_i^{l'}; 0 \le l' < l-\e, 1 \le i \le r][(X_j^0)\i; j \in J]$ for $l > \e$.
 \end{itemize}

The following results are immediate by definition.
 \begin{lemma} \label{CR}
     We have the following properties.

     (1) $\CR^{\ge \e} \CR^{\ge \e'} \subseteq \CR^{\ge \e + \e'}$ and $\CR^{\ge \e} \CR^{> \e'} \subseteq \CR^{> \e + \e'}$ for $\e, \e' \in \BR$.

     (2) $d \CR^{\ge \e} = \CR^{\ge \e+m}$ and $d \CR^{>\e} = \CR^{> \e+m}$ for $d \in \varpi^m \CO_\brk^\times$ with $m \in \BZ$.

     (3)  $f_j\i \in \CR^{\ge -\<\l, \b_j\>}$ for $j \in J$, where $f_j = \varpi^{\<\l, \b_j\>} \sum_{l=0}^\infty X_j^l \varpi^l \in \CR^{\ge \<\l, \b_j\>}$.
\end{lemma}

We set $\bU_\a(\CR)^{\ge \l} = \bU_\a(\CR^{\ge \<\l, \a\>})$ for $\a \in \Phi$ . It follows from Lemma \ref{CR} that \[ [\bU_\a(\CR)^{\ge \l}, \bU_\b(\CR)^{\ge \l}] \subseteq \prod_{i, j \ge 1} \bU_{i\a + j\b}(\CR)^{\ge \l} \ \text{ for } -\a \neq \b \in \Phi.\] So for $w \in W$ we can define the following subgroup \[ \bU_{w\s}(\CR)^{\ge \l} = \prod_{\g \in \Phi_{w\s}^+} \bU_\g(\CR)^{\ge \l} \subseteq \bU_{w\s}(\CR).\] Moreover, we can define $\bU_\g(\CR)^{>\l}$ and $\bU_{w\s}(\CR)^{>\l}$ in a similarly way. Note that $\bU_{w\s}(\CR)^{>\l}$ is a normal subgroup of $\bU_{w\s}(\CR)^{\ge \l}$.

Note that each element of $\bU(\CR)$ defines a map from $H_{c\s, J}^\l$ to $\bU(\brk)$ in a natural way.
\begin{lemma} \label{surj}
    Let $g \in h \bU_{c\s}(\CR)^{>\l}$ with $h = \prod_{i=1}^r \bU_{\b_i}(f_i)$, where \[f_i = \varpi^{\<\l, \b_i\>} \sum_{l=0}^\infty X_i^l \varpi^l \text{ for } 1 \le i \le r.\]  Then $g$ induces a surjective endomorphism of $H_{c\s, J}^\l$.
\end{lemma}
\begin{proof}
    Write $g = \prod_{i=1}^r \bU_{\b_i}(g_i)$ with $g_i = \varpi^{\<\l, \b_i\>} \sum_{l=0}^\infty g_i^l \varpi^l \in \CR$. By assumption and Lemma \ref{commute} we have

    (a) $g_i^0 = X_i^0$;

    (b) $g_i^l = X_i^l + \d_i^l$ with $\d_i^l \in \overline \BF_q[X_{i'}^{l'}; 0 \le l' < l, \ 1 \le i' \le r][(X_j^0)\i; j \in J]$.
    By (a) we have $g(H_{c\s, J}^\l) \subseteq H_{c\s, J}^\l$. It remains to show that $g(H_{c\s, J}^\l) = H_{c\s, J}^\l$.

    Let $y = (y_i^l) \in H_{c\s, J}^\l$. We construct a point $x = (x_i^l) \in H_{c\s, J}^\l$ inductively as follows. If $l = 0$, we set $x_i^l = y_i^l$. Suppose that the point $(x_i^{l'})_{l' < l}$ is already constructed. In view of (b) we set \[x_i^l = y_i^l - \d_i^l((x_i^{l'})_{l' < l})\] and this finishes the construction of $x$. It follows from (a) and (b) that $y_i^l = g_i^l(x)$ and hence $g(x) = y$. So $g(H_{c\s, J}^\l) = H_{c\s, J}^\l$ as desired.
\end{proof}

\subsection{} \label{cross-subsec}
Let $\g \in \Phi^+$. Following \cite{M} we define \[\Cross_{c\s}(\g) = \Phi^+ \cap c\s(\BZ_{\ge 1} \g + \BZ_{\ge 0} \Phi_{c\s}^+).\] For $\G \subseteq \Phi^+$ we set $\Cross_{c\s}(\G) = \cup_{\g \in \G} \Cross_{c\s}(\g)$.

The following result is proved in Theorem B and Lemma 2.37 of \cite{M}.
\begin{theorem} \label{cross}
    We have $\Cross_{\dc\s}^d(\Phi^+) = \emptyset$ for $d \gg 0$. Here $\Cross_{\dc\s}^d$ denotes the $d$-fold composition of $\Cross_{\dc\s}$.
\end{theorem}

Let $\CR$ be the $\brk$-algebra as in \S \ref{R-sec}. For a lift $\dc \in \bN(\brk)$ of $c$ we define \[\Xi_{\dc\s}: \bU(\CR) \times \bU_{c\s}(\CR) \to \bU(\CR) \times \bU_{c\s}(\CR), \ (x, y) \mapsto (x', y')\] such that $x' \dc\s y' = \dc\s y x$.
\begin{corollary} \label{conj}
      The image of $\pr_1 \circ \Xi_{\dc\s}^d$ is trivial for $d \gg 0$. Here $\Xi_{\dc\s}^d$ is the $d$-fold composition of $\Xi_{\dc\s}$, and $\pr_1: \bU(\CR) \times \bU_{c\s}(\CR) \to \bU(\CR)$ is the natural projection.
\end{corollary}
\begin{proof}
    For $d \in \BZ_{\ge 0}$ let $\bU_d \subseteq \bU$ be the subgroup generated by $\bU_\a$ for $\a \in \Cross_{c\s}^d(\Phi^+)$. By Theorem \ref{cross}, it suffices to show that \[\Xi_{\dc\s}(\bU_d(\CR) \times \bU_{c\s}(\CR)) \subseteq \bU_{d+1}(\CR) \times \bU_{c\s}(\CR),\] that is, $\dc\s \bU_{c\s}(\CR) \bU_\a(\CR) \subseteq \bU_{d+1}(\CR) \dc\s \bU_{c\s}(\CR)$ for $\a \in \Cross_{c\s}^d(\Phi^+)$. Let \[\G = \BZ_{\ge 1} \a + \BZ_{\ge 0}\Phi_{c\s}^+ \subseteq \Phi^+.\] Note that there is decomposition \[\prod_{\g \in \G} \bU_\g = \prod_{\g \in \G \cap \Phi_{w_0 c\s}^+} \bU_\g  \prod_{\g \in \G \cap \Phi_{c\s}^+} \bU_\g\] of unipotent subgroups of $\bG$. For $x \in \bU_\a(\CR)$ and $y \in \bU_{c\s}(\CR)$ we have \[y x y\i \in \prod_{\g \in \G} \bU_\g(\CR).\] Write $y x y\i= x' y'$, where $x' \in \prod_{\g \in \G \cap \Phi_{w_0 c\s}^+} \bU_\g(\CR)$ and $y' \in \prod_{\g \in \G \cap \Phi_{c\s}^+} \bU_\g(\CR)$. Thus \[\dc\s y x = {}^{\dc\s} (x') \dc\s y' y,\] where $y' y \in \bU_{c\s}(\CR)$ and \[{}^{\dc\s} (x') \in \prod_{\g \in \Phi^+ \cap {}^{c\s} \G} \bU_\g(\CR) = \prod_{\g \in \Cross_{c\s}(\a)} \bU_\g(\CR) \subseteq \bU_{d+1}(\CR)\] as desired.
\end{proof}

\section{Proof of Theorem \ref{st-cross}}
First we show that the main result is independent of the choice of minimal $\s$-Coxeter elements.

Recall that two element $w, w' \in W$ are equivalent by $\s$-cyclic shifts if there exists a sequence \[w=w_0, w_1, \dots, w_n = w'\] such that for each $1 \le i \le n$ the elements $w_{i-1}, w_i$ have the same length, and are $\s$-conjugate by a simple reflection.
\begin{lemma} \label{shift}
    Let $c'$ be another minimal $\s$-Coxeter element. Then Theorem \ref{st-cross} holds for $c$ if and only if it holds for $c'$.
\end{lemma}
\begin{proof}
    Note that all minimal $\s$-Coxeter elements are equivalent by $\s$-cyclic shifts. So we can assume that $c' = s_{\a_1} c \s(s_{\a_1})$. Suppose that $\dc \in \varpi^\mu \bN(\COk)$ and let $\dc' = \tilde s_{\a_1}\i \dc \s(\tilde s_{\a_1}) \in \varpi^{s_{\a_1}(\mu)} \bN(\COk)$. Note that there are natural isomorphisms \begin{align*}\dc \bT(\COk) \bU_c(\brk) &\cong \varpi^\mu \bT(\COk) \times \tilde s_{\a_1} \bU_{\a_1}(\brk) \times \cdots \times \tilde s_{\a_r} \bU_{\a_r}(\brk) \\ &\cong \tilde s_{\a_1} \bU_{\a_1}(\brk) \times \varpi^{s_{\a_1}(\mu)} \bT(\COk) \times \tilde s_{\a_2} \bU_{\a_2}(\brk) \times \cdots \times \tilde s_{\a_r} \bU_{\a_r}(\brk). \end{align*} For $z \in \varpi^{s_{\a_1}(\mu)} \bT(\COk)$ and $u_i \in \bU_{\a_i}(\brk)$ the map \[\tilde s_{\a_1} u_1 z \tilde s_{\a_2} u_2  \cdots \tilde s_{\a_r} u_r \mapsto z \tilde s_{\a_2} u_2 \cdots \tilde s_{\a_r} u_r {}^\s(\tilde s_{\a_1} u_1)\] induces a bijection \[\psi: \dc \bT(\COk) \bU_c(\brk) \cong \dc' \bT(\COk) \bU_{c'}(\brk).\] Thus $\psi([b] \cap \dc \bT(\COk) \bU_c(\brk)) = [b] \cap \dc' \bT(\COk) \bU_{c'}(\brk)$. Now the statement follows from the following equality \begin{align*}
        & \quad\ \psi(\dc \bT(\COk) H_{c, \mu, b}) \\
        &= \psi(\varpi^\mu \bT(\COk) \tilde s_{\a_1} H_{\mu, b, \a_1} \tilde s_{\a_2} H_{\mu, b, \a_2} \cdots \tilde s_{\a_r} H_{\mu, b, \a_r}) \\
        &= \varpi^{s_{\a_1}(\mu)} \bT(\COk) \tilde s_{\a_2} H_{\mu, b, \a_2} \cdots \tilde s_{\a_r} H_{\mu, b, \a_r} \s(\tilde \a_1) {}^{\s \varpi^{s_{\a_1}(\mu)}}H_{\mu, b, \a_1} \\
        &= \varpi^{s_{\a_1}(\mu)} \bT(\COk) \tilde s_{\a_2} H_{s_{\a_1}(\mu), b, \a_2} \cdots \tilde s_{\a_r} H_{s_{\a_1}(\mu), b, \a_r} \tilde s_{\s(\a_1)} H_{s_{\a_1}(\mu), b, \s(\a_1)} \\
        &= \dc' \bT(\COk) H_{c', s_{\a_1}(\mu), b}, \end{align*} where the third equality follows from that \begin{align*} \<s_{\a_1}(\mu) - \nu(b), \o_i\> = \begin{cases} \<\mu - \nu(b), \o_i\>  - \<\mu, \a_i\>, &\text{ if } i = 1; \\ \<\mu -\nu(b), \o_i\>, &\text{ otherwise.} \end{cases} \end{align*} The proof is finished.
\end{proof}

\begin{lemma} \label{preserve}
     Let $\mu, b, \l$ be as in Theorem \ref{st-cross}. Let $\CR$ be as in \S \ref{R-sec} associated to the pair $(\l, J_{\nu(b)})$. Then the map $\Xi_{\dc\s}$ in \S \ref{cross-subsec} induces endomorphism of (1) $\bU(\CR)^{\ge \l} \times \bU_{c\s}(\CR)^{\ge \l}$ and (2) $\bU(\CR)^{>\l} \times h \bU_{c\s}(\CR)^{>\l}$ for $h \in \bU_{c\s}(\CR)^{\ge \l}$.
\end{lemma}
\begin{proof}
    By the proof of Corollary \ref{conj} and that $\bU_{c\s}(\CR)^{\ge \l}$ normalizes $\bU_{c\s}(\CR)^{>\l}$, it suffices to show that \[{}^{\dc\s} \bU_{w_0 c\s}(\CR)^{>\l} \subseteq \bU(\CR)^{>\l} \text{ and } {}^{\dc\s} \bU_{w_0 c\s}(\CR)^{\ge \l} \subseteq \bU(\CR)^{\ge \l},\] where $w_0$ is the longest element of $W$. We only show the first inclusion. Indeed, for $\g \in \Phi_{w_0 c \s}^+$, that is, $c\s(\g) \in \Phi^+$, we have \begin{align*}{}^{\dc\s} \bU_\g(\CR^{> \<\l, \g\>}) &= \bU_{c\s(\g)}(\CR^{>\<\mu, c\s(\g)\> + \<\l, \g\>}) = \bU_{c\s(\g)}(\CR^{>\<\l, c\s(\g)\> + \<\nu(b), c\s(\g)\>}) \\ &\subseteq \bU_{c\s(\g)}(\CR^{>\<\l, c\s(\g)\>}),\end{align*} where the second equality follows from Lemma \ref{eq} (1), and the inclusion follows from that $\<\nu(b), c\s(\g)\> \ge 0$ since $c\s(\g) \in \Phi^+$. So the statement follows.
\end{proof}

Let $A, B \subseteq \bG(\brk)$ and let $C$ be a subgroup of $\bG(\brk)$. We write $A\s \sim_{C} B\s$ if $C \cdot_{\s} A = C \cdot_{\s} B$.
\begin{proof}[Proof Proposition \ref{st-cross}]
     As $\mu - \nu(b) = \l - c\s(\l)$ and $\l \in X_*(\bT)$ it follows that $\nu(b) \in X_*(\bT)$. Thus we may assume $b = \varpi^\eta$ for some $\eta \in X_*(\bT)$ (for example $\eta = \nu(b)$) such that the $\s$-average $\eta^\diamond$ of $\eta$ is dominant. Set $\nu = \nu(b) = \eta^\diamond$ and $J = J_{\nu(b)}$. By Lemma \ref{shift} we may assume $\<\nu, \a_1\> \neq 0$ if $\nu$ is non-central.

     By Lemma \ref{unique} we may assume further that $\mu - \eta = \sum_{i=1}^r n_i \a_i^\vee$ with $n_i \in \BZ$ for $1 \le i \le r$. By Proposition \ref{red}, it is equivalent to show that \[\varpi^\eta \bT(\COk) K_{\mu - \eta} \s \sim_{\bU(\brk)} \dc\s \bT(\COk) H_{c\s, J}^\l.\] We argue by induction on $r$. If $r = 0$, the statement is trivial. Suppose $r \ge 1$. Let $\bM$ be the standard Levi subgroup generated by $\bT$ and the root subgroup $\bU_{\pm \a}$ for $\a \in \Pi - C_1$, where $C_1 \subseteq \Pi$ is the $\s$-orbit of $\a_1$. Let $c_\bM = s_{\a_1} c$ and let $\dc_{\bM} \in \varpi^{\mu - n_1\a_1^\vee} \bN(\COk)$ be a lift of $c_\bM$. By Lemma \ref{eq}, \[\<\l, \b_i\> = \<\mu - \nu, \o_i\> = \<\mu - \eta^\diamond, \o_i\> = \<\mu-\eta, \o_i\> = n_i \text{ for } 1 \le i \le r.\] Hence (using that $\mu - \nu = \l - c\s(\l)$) we have \[\mu-n_1\a^\vee - \nu = \l - c_\bM\s(\l).\] By induction hypothesis and Proposition \ref{red} (where we take $(\bG, \dc, b) = (\bM, \dc_\bM, \varpi^\eta)$, and note that $\k_\bM(b) = \k_\bM(\dc_\bM)$ and $\nu_\bM(b) = \nu$), \[\tag{a} \varpi^\eta  \bT(\COk) K_{\mu-n_1\a_1^\vee - \eta}^{\bM} \s \sim_{\bU_\bM(\brk)} \dc_\bM \s \bT(\COk) H_{c_{\bM}\s, J^\bM}^{\l, \bM},\] where $\bU_\bM = \bU \cap \bM$, $J^\bM = J - \{1\}$, and $K_{\mu-n_1\a_1^\vee - \eta}^{\bM}$, $ H_{c_{\bM}\s, J^\bM}^{\l, \bM}$ are defined similarly for $\bM$. Note that \[K_{\mu-\eta} = \bU_{-\a_1}(\varpi^{m_1} \CO_\brk^\times) K_{\mu-n_1\a_1^\vee-\eta}^\bM,\] where $m_1 = -n_1 - \sum_{i=2}^r n_i \<\a_i^\vee, \a_1\>$. Thus \begin{align*}
        &\quad\ \varpi^\eta \bT(\COk) K_{\mu-\eta} \s \\
        &=\varpi^\eta \bT(\COk) \bU_{-\a_1}(\varpi^{m_1} \CO_\brk^\times) K_{\mu-n_1\a_1^\vee-\eta}^\bM \s \\
        &=\bU_{-\a_1}(\varpi^{m_1 - \<\eta, \a_1\>} \CO_\brk^\times) \varpi^\eta \bT(\COk) K_{\mu - n_1\a_1^\vee-\eta}^\bM \s \\
        &\sim_{\bU_\bM(\brk)} \bU_{-\a_1}(\varpi^{m_1 - \<\eta, \a_1\>} \CO_\brk^\times) \dc_\bM \s \bT(\COk) H_{c_\bM\s, J^\bM}^{\l, \bM} \\
        &=\bU_{-\a_1}(\varpi^{m_1 - \<\eta, \a_1\>}  \CO_\brk^\times) \varpi^{\mu - n_1\a_1^\vee} \bT(\COk) \tilde s_{\a_2} \cdots \tilde s_{\a_r} \s H_{c_\bM\s, J^\bM}^{\l, \bM} \\
        &=\varpi^{\mu - n_1\a_1^\vee} \bT(\CO) \bU_{-\a_1}(\varpi^{-n_1} \CO_\brk^\times) \tilde s_{\a_2} \cdots \tilde s_{\a_r} \s H_{c_\bM\s, J^\bM}^{\l, \bM} \\
        &=\cup_{y \in \CO_\brk^\times} \varpi^{\mu - n_1\a_1^\vee} \bT(\CO) (\varpi^{-n_1} y\i)^{-\a_1^\vee}  \bU_{\a_1}(\varpi^{-n_1} y\i) \tilde s_{\a_1} \bU_{\a_1}(\varpi^{n_1} y) \tilde s_{\a_2} \cdots \tilde s_{\a_r} \s H_{c_\bM\s, J^\bM}^{\l, \bM} \\
        &=\cup_{y \in \CO_\brk^\times} \varpi^{\mu} \bT(\CO) \bU_{\a_1}(\varpi^{-n_1} y\i) \tilde s_{\a_1} \bU_{\a_1}(\varpi^{n_1} y) \tilde s_{\a_2} \cdots \tilde s_{\a_r} \s  H_{c_\bM\s, J^\bM}^{\l, \bM} \\
        &= \cup_{z \in \bT(\COk)} \cup_{y \in \CO_\brk^\times} \bU_{\a_1}(d_z \varpi^{\<\mu, \a_1\> - n_1} y\i) \varpi^\mu z \tilde s_{\a_1} \cdots \tilde s_{\a_r} \s \bU_{\b_1}(\varpi^{n_1} y) H_{c_\bM\s, J^\bM}^{\l, \bM},
    \end{align*} where the relation $\sim_{\bU_M(\brk)}$ follows from (a) and the observation that $\bU_\bM$ commutes with $\bU_{-\a_1}$, the fourth equality follows from the equality $\mu - \eta = \sum_{i=1}^r n_i \a_i^\vee$, and $d_z \in \CO_\brk^\times$ is certain constant depending on $z \in \bT(\COk)$. Therefore, it suffices to show
    \[ \tag{$\ast$} \dc\s H_{c\s, J}^\l \sim_{\bU(\brk)} \cup_{y \in \CO_\brk^\times} \bU_{\a_1}(d \varpi^{\<\mu, \a_1\> - n_1} y\i) \dc\s \bU_{\b_1}(\varpi^{n_1} y) H_{c_\bM\s, J^\bM}^{\l, \bM},\]
    where $d \in \CO_\brk^\times$ and $\dc \in \varpi^\mu \bN(\COk)$ is a lift of $c$.

    \

    Set \[V_{c\s, J}^\l = \bU_{\b_1}(\varpi^{n_1} \CO_\brk^\times) H_{c_\bM\s, J^\bM}^{\l, \bM} \subseteq H_{c\s, J}^\l.\] Let $\CR$ be as in \S \ref{R-sec} associated to the pair $(J, \l)$. As $n_1 = \<\l, \b_1\>$, the right hand side of ($\ast$) is the image of map \[h^0 = (h_1^0, h_2^0) : V_{c\s, J}^\l \to \bU(\brk) \times \bU_{c\s}(\brk),\] where $h_1^0 = \bU_{\a_1}(d \varpi^{\<\mu, \a_1\>} f_1\i)$, $h_2^0 = \prod_{i=1}^r \bU_{\b_i}(f_i)$, and \[f_i = \varpi^{\<\l, \b_i\>} \sum_{k=0}^\infty X_i^l \varpi^l \in \CR^{\ge \<\l, \b_i\>}.\] Let $h^n = (h_1^n, h_2^n) = \Xi_{\dc\s}^n \circ h^0$ for $n \in \BZ$. By Corollary \ref{conj}, there exists $N \gg 0$ such that $h_1^N$ is trivial. Hence ($\ast$) is equivalent to that the image $\Im(h_2^N)$ of the map $h_2^N: V_{c\s, J}^\l \to \bU_{c\s}(\brk)$ equals $H_{c\s, J}^\l$.

    Note that $f_1\i \in \CR^{\ge -\<\l, \b_1\>}$ by Lemma \ref{CR} (3). Moreover, It follows from Lemma \ref{eq} (3) that \[\<\mu, \a_1\> - \<\l, \b_1\> - \<\l, \a_1\> = \<\nu, \a_1\> \ge 0.\] So $h_1^0 \in \bU(\CR)^{\ge \l}$, and $h_1^0 \in \bU(\CR)^{>\l}$ if $\<\nu, \a_1\> > 0$.

    \

    First we assume that $\<\nu, \a_1\> \neq 0$. In this case, $h^0 \in \bU(\CR)^{>\l} \times h_2^0 \bU_{c\s}(\CR)^{> \l}$ and $V_{c\s, J}^\l = H_{c\s, J}^\l$. It follows from Lemma \ref{preserve} (2) that $h_2^N \in h_2^0 \bU_{c\s}^{>\l}$. By Lemma \ref{surj} we have $\Im(h_2^N) = H_{c\s, J}^\l$ as desired.

    Now we assume that $\<\nu, \a_1\> = 0$. Then $\nu$ is central by assumption. It follows from Lemma \ref{preserve} (1) that $h_2^N \in \bU_{c\s}(\CR)^{\ge \l}$. Hence \[\Im(h_2^N) \subseteq {}^{\varpi^\l} \bU_{c\s}(\COk) = H_{c\s, J}^\l.\] By Proposition \ref{red} this means that \[[b] \cap \dc \bU_c(\brk) \subseteq \dc H_{c, \mu, b}.\] On the other hand, as $\nu$ is central and $\mu + c\s(\l) = \l + \nu$, we see that $\dc \s$ fixes the Moy-Prasard group $H$ generated by $\bT(\COk)$ and $\bU_\a(\varpi^{\<\l, \a\>} \COk)$ for $\a \in \Phi$. Since ${}^{\s\i} H_{c, \mu, b} = H_{c\s, J}^\l \subseteq H$, it follows from Lang's theorem that \[\dc H_{c, \mu, b} \subseteq [\dc] \cap \dc \bU_c(\brk) = [b] \cap \dc \bU_c(\brk).\] So $[b] \cap \dc \bU_c(\brk) = \dc H_{c, \mu, b}$ and the theorem also holds.
\end{proof}

\section{Applications}
Let $\mu \in X_*(\bT)$ and let $\dc \in \varpi^\mu \bN(\COk)$ be a lift of $c$.

\subsection{} \label{subsec:closure}
First we recall the dominance order $\leq$ on $B(\bG)$. Namely, $[b'] \le [b] \in B(\bG)$ if $\k(b) = \k(b')$ and $\nu(b) - \nu(b') \in \sum_{\a \in \Pi} \BR_{\ge 0} \a^\vee$.

Now we prove Theorem \ref{closure}. By Theorem \ref{main} we have \[ \tag{a} \overline{[b] \cap \dc\bU_c(\brk)} = \varpi^\mu \prod_{i=1}^r \ts_{\a_i} \bU_{\a_i}(\varpi^{\lceil\<\mu - \nu(b), \o_i\> \rceil } \COk).\] Then it follows that $\overline{[b] \cap \dc \bU(\brk)} - ([b] \cap \dc \bU(\brk))$ is pure of codimension one having $\sharp J_{\nu(b)}$ irreducible components.

Let $[b'] \le [b]$. Then $\<\mu-\nu(b), \o_i\> \le \<\mu-\nu(b'), \o_i\>$ for $1 \le i \le r$. By (a) we have \[\overline{[b'] \cap \dc \bU(\brk)} \subseteq \overline{[b] \cap \dc \bU(\brk)}.\] On the other hand, Let $[b''] \in B(\bG)$ which intersects $\overline{[b] \cap \dc \bU(\brk)}$. Then $\k(b'') = \k(b)$. Moreover, it follows from Theorem \ref{main} that \[ \tag{a} \<\mu-\nu(b), \o_i\> \le \<\mu-\nu(b''), \o_i\> \text{ for } i \in J_{\nu(b'')}.\] Let $\Pi_1, \Pi_2 \subseteq \Pi$ be $\s$-stable subsets (which may be empty) such that $\Pi = \Pi_1 \sqcup \Pi_2$ and \[\nu(b) - \nu(b'') = v_1 - v_2,\] where $v_1 \in \sum_{\a \in \Pi_1} \BR_{\ge 0}\a^\vee$ and $v_2 \in \sum_{\a \in \Pi_2} \BR_{> 0}\a^\vee$. It follows from (a) that $\<\nu(b''), \a\> = 0$ for $\a \in \Pi_2$. Let $\rho_2$ be the half sum of roots in $\Phi^+$ spanned by $\Pi_2$. Then $\<\nu(b''), \rho_2\> = 0$ and $\<\a^\vee, \rho_2\> > 0$ for $\a \in \Pi_2$. We have \[0 \le \<v_2, \rho_2) = \<v_1, \rho_2\> - \<\nu(b), \rho_2\> + \<\nu(b''), \rho_2\> \le 0.\] Thus $v_2 = 0$ and hence $[b''] \le [b]$ as desired.

\subsection{} Now we prove Theorem \ref{poset}. Let \[[b'] = [b_0] < [b_1] < \cdots < [b_n] = [b]\] be a maximal chain in the poset $(B(\bG), \leq)$. We claim that

(a) $[b_{i-1}] \cap \dc \bU(\brk)$ is of codimension one in $\overline{[b_i] \cap \dc \bU(\brk)}$ for $1\le i \le n$

By Theorem \ref{closure}, there is $[b_i'] < [b_i]$ such that $\overline {[b_i'] \cap \dc \bU(\brk)}$ is an irreducible component of $\overline{[b_i] \cap \dc \bU(\brk)} - ([b_i] \cap \dc \bU(\brk))$ containing $[b_{i-1}] \cap \dc \bU(\brk)$. In particular, $[b_i'] \cap \dc \bU(\brk)$ is of codimension one in $\overline{[b_i] \cap \dc \bU(\brk)}$, and $[b_{i-1}] \leq [b_i'] < [b_i]$. Hence $[b_i'] = [b_{i-1}]$ since the chain $[b_{i-1}] \leq [b_i]$ is maximal. So (a) is proved.

By (a) it follows that the length of any maximal chain between $[b']$ and $[b]$ equals the codimension of $[b'] \cap \dc \bU(\brk)$ in $\overline{[b] \cap \dc \bU(\brk)}$. So the poset $(B(\bG), \leq)$ is ranked, and in view of \S \ref{subsec:closure} the length function is given by \begin{align*}\leng([b], [b']) &= \codim ([b'] \cap \dc \bU(\brk), \overline{[b] \cap \dc \bU(\brk)}) \\ &=  \sum_{i=1}^r \lceil \<\mu-\nu(b'), \o_i\>\rceil - \lceil \<\mu-\nu(b), \o_i\>\rceil,\end{align*} as desired.

\subsection{} \label{subsec:identity}
For $\mu \in X_*(\bT)$ recall that \[B(\bG, \mu)_\irr = \{[b] \in B(\bG, \mu); \nu(\varpi^\mu)-\varpi(b) \in \sum_{\a \in \Pi} \BR_{>0} \a^\vee\}.\] First we show the following result.
\begin{lemma} \label{indec}
    Let $\mu \in X_*(\bT)$ and let $\dc \in \varpi^\mu \bN(\COk)$ be a lift of $c$. Suppose that $B(\bG, \mu)_\irr \neq \emptyset$. Then \[
    \dc \bU_c(\varpi \COk) = \bigsqcup_{[b] \in B(\bG, \mu)_\irr} [b] \cap \dc\bU_c(\brk).\]
\end{lemma}
\begin{proof}
    We may assume $\mu$ is dominant. Then $\nu(\varpi^\mu)$ equal the $\s$-average $\mu^\diamond$ of $\mu$. Note that $B(\bG, \mu)_\irr \neq \emptyset$ if and only if $\mu^\diamond$ is non-central on each $\s$-orbit of connected components of the Dynkin diagram of $\Pi$.

    If $[b] \in B(\bG, \mu)_\irr$, by definition \[\<\mu-\nu(b), \o_i\> = \<\mu^\diamond - \nu(b), \o_i\> > 0 \text{ for } 1 \le i \le r.\] So it follows from Theorem \ref{main} that $[b] \cap \dc\bU_c(\brk) \subseteq \dc\bU(\varpi\COk)$.

    On the other hand, let $[b'] \in B(\bG, \mu) - B(\bG, \mu)_\irr$ which intersects $\dc\bU_c(\varpi\COk)$. Then there exist a proper subset $\s$-stable $\Pi_1 \subsetneq \Pi$ and $v \in \sum_{\a \in \Pi_1} \BR_{>0} \a^\vee$ such that $v = \mu^\diamond - \nu(b')$. By Theorem \ref{main} we have $\<\nu(b'), \a\> = 0$ for $\a \in \Pi - \Pi_1$. Let $\rho_2$ be the half sum of roots in $\Phi^+$ spanned by $\Pi - \Pi_1$. Then \[ 0 \le \<\mu^\diamond, \rho_2\> = \<\nu(b'), \rho_2\> + \<v, \rho_2\> = \<v, \rho_2\> \le 0.\] Thus $\<\mu^\diamond, \rho_2\> = \<v, \rho_2\> = 0$. This means that $\mu^\diamond$ is central on $\Pi - \Pi_1$, and $\<\a^\vee, \b\> = 0$ for $\a \in \Pi_1$ and $\b \in \Pi - \Pi_1$. Therefore, $\Pi-\Pi_1$ is a union of $\s$-orbits of connected components of the Dynkin diagram of $\Pi$. This contradicts that $\mu^\diamond$ is non-central on any $\s$-orbit of connected components of the Dynkin diagram of $\Pi$. The proof is finished.
\end{proof}

Now we are ready to prove Theorem \ref{identity}. We can assume  $\mu$ is dominant. Let $m$ be a sufficiently large integer such that $m > \<\mu, \o_i\>$ for $1 \le i \le r$. Suppose that the residue field of $k$ has $q$ elements. Then for any $[b] \in B(\bG, \mu)$ it follows from Theorem \ref{main} and Theorem \ref{poset} that \[\tag{a} \sharp ( ([b] \cap \dc \bU_c(k)) / \bU_c(\varpi^m \CO_k)) = (q-1)^{\sharp J_{\nu(b)}} q^{mr - \sharp J_{\nu(b)} - \leng([b], [\varpi^\mu])},\] where $\CO_k$ is the integer ring of $k$.

Suppose $B(\bG, \mu)_\irr \neq \emptyset$. Combining (a) with Lemma  \ref{indec} we have \begin{align*}
q^{(m-1)r} &= \sharp (\dc \bU_c(\varpi \CO_k) / \bU_c(\varpi^m \CO_k)) \\ &= \sum_{[b] \in B(\bG, \mu)_\irr} \sharp ( ([b] \cap \dc \bU_c(k)) / \bU_c(\varpi^m \CO_k)) \\ &= \sum_{[b] \in B(\bG, \mu)_\irr} (q-1)^{\sharp J_{\nu(b)}} q^{mr - \sharp J_{\nu(b)} - \leng([b], [\varpi^\mu])}.
\end{align*} Thus, $\sum_{[b] \in B(\bG, \mu)_\irr} (q-1)^{\sharp J_{\nu(b)}} q^{r - \sharp J_{\nu(b)} - \leng([b], [\varpi^\mu])} = 1$ and the second identity follows.

Recall that $\nu(\varpi^\mu)$ equals the $\s$-average $\mu^\diamond$ of $\mu$. Then $\<\mu - \nu(\varpi^\mu), \o_i\> = \<\mu - \mu^\diamond, \o_i\> = 0$ for $1 \le i \le r$. Therefore, \[\dc\bU_c(\COk) = \overline{[\varpi^\mu] \cap \dc\bU_c(\brk)} = \bigsqcup_{[b] \in B(\bG, \mu)} [b] \cap \dc\bU_c(\brk).\] Now the first identity follows in a similar way as above.

\end{document}